\documentclass{amsart}
\usepackage{amsaddr}

\usepackage[utf8]{inputenc}
\usepackage[T1]{fontenc}
\usepackage[english]{babel}
\usepackage{cite}
\usepackage{hyperref}
\usepackage{subcaption}
\usepackage{graphicx} 
\usepackage{geometry}

\usepackage{amsmath,amssymb}
\usepackage{amsthm}
\usepackage{todonotes} 
\usepackage{enumitem}

\newtheorem{thm}{Theorem}

\newtheorem{lem}[thm]{Lemma}
\newtheorem{cor}[thm]{Corollary}
\newtheorem{proposition}[thm]{Proposition}
\newtheorem{rem}[thm]{Remark}
\newtheorem{conj}[thm]{Conjecture}

\newcommand{\ex}{{\rm ex}}

\newcommand{\cL}{{\mathcal L}}

\newcommand{\E}{\mathbb{E}}

\begin{document}
	
\title{Splits with forbidden subgraphs}
\author{Maria Axenovich}
\address{Karlsruhe Institute of Technology, Karlsruhe, Germany}
\email{maria.aksenovich@kit.edu}
\author{Ryan R. Martin}
\address{Iowa State University, Ames, Iowa, USA}
\email{rymartin@iastate.edu}
\date{\today}


\maketitle
	
\begin{abstract}
	In this paper, we fix a graph $H$ and ask into how many vertices each vertex of a clique of size $n$ can be ``split'' such that the resulting graph is $H$-free. Formally: A graph is an $(n,k)$-graph if its vertex set is a pairwise disjoint union of $n$ parts of size at most $k$, such that there is an edge between any two distinct parts. Let 
	$$ f(n,H) 	= 	\min \{k \in \mathbb N : \mbox{there is an $(n,k)$-graph $G$ such that $H\not\subseteq G$}\} . $$
	Barbanera and Ueckerdt \cite{Ba} observed that $f(n, H)=2$ for any graph $H$ that is not bipartite. 
	If a graph $H$ is bipartite and  has a well-defined Tur\'an exponent, i.e., ${\rm ex}(n, H) = \Theta(n^r)$ for some $r$, we show that $\Omega (n^{2/r-1}) = f(n, H) = O (n^{2/r -1} \log ^{1/r} n)$.
	We extend this result to all bipartite graphs for which upper and a lower Tur\'an exponents do not differ by much. In addition, we prove that 
	$f(n, K_{2,t}) =\Theta(n^{1/3})$ for any fixed integer $t\geq 2$.
\end{abstract}
	
\section{Introduction}
	
For positive integers $n$ and $k$, we say that a graph $G$ is an {\it $(n,k)$-graph} if $G$ is an $n$-partite graph with exactly one edge between any two distinct parts, also referred to as \textit{blobs}, such that each blob has size at most $k$.  Note that contracting each blob of an $(n,k)$-graph into a single vertex results in a complete graph on $n$-vertices, $K_n$.  Thus an $(n,k)$-graph is sometimes referred to as a $k$-{\it split} of $K_n$.   In general, a $k$-split of a graph $H$ on a vertex set $\{v_1, \ldots, v_m\}$ is a graph $F$ whose vertex set is a pairwise disjoint union of sets $V_1, \ldots, V_m$, each of size at most $k$, such that there is an edge between $v_i$ and $v_j$ if and only if there is an edge between $V_i$ and $V_j$, $1\leq i<j\leq m$. 

A study of graph splits traces back to Heawood \cite{H}, who  was interested in the smallest $k$ such that $H$ has a $k$-split which is a planar graph. This is the so-called planar  split thickness of $H$, see also Eppstein et al. \cite{E}. Heawood proved that the planar split thickness of a clique on $12$ vertices is $2$.  In addition, the notion of splits was extensively used in bar visibility representations of graphs. In particular, it is known that if a graph $H$ has a $k$-split that is planar and two-connected, then $H$ has a bar visibility representation with $k$ bars per vertex, see a number of papers by West, Hutchinson and others \cite{W1, W2, W3, W4}. 

While most of the existent literature on splits concerns planarity, here we address a more general question about splits that belong to a given monotone family.
This question for complete graphs was asked by Ueckerdt and the study was initiated in a Master's thesis by Barbanera \cite{Ba}. They defined the following function for a graph $H$:
\begin{align*}
	f(n,H) 	= 	\min \{k \in \mathbb N : \mbox{there is an $(n,k)$-graph $G$ such that $H\not\subseteq G$}\} .
\end{align*}
It is clear that $f(n, H)$ is well defined as long as $\Delta(H)$, the maximum degree of $H$, is at least $2$. Indeed,  we can always take $k=n-1$, and consider an $(n,n-1)$-graph that is a perfect matching. This graph contains no graph of maximum degree greater than one. So, $f(n, H)\leq n-1$ if $\Delta(H)\geq 2$. Moreover, if $n<|V(H)|$, then $f(n,H)=1$. In this paper, we give better bounds on $f(n,H)$. The case where  $H$ is not bipartite is completely resolved and we provide the proof here for completeness. 
	
\begin{proposition}[Barbanera, Ueckerdt~\cite{Ba}]
	If $H$ is a non-bipartite graph and  $n\geq |V(H)|$, then $f(n, H) = 2$.
\end{proposition}
	
\begin{proof}
	Construct $G$, an $(n,2)$-graph with blobs of size $2$ as follows: For each blob, color one vertex red and another blue. 
For every pair of blobs, choose an edge between that contains one red vertex and one blue vertex.  The resulting $(n,2)$-graph is bipartite and thus  does not contain $H$ as a subgraph. 
\end{proof}

In this paper, we show that $f(n, H)$ for bipartite graphs is closely connected to the extremal function $\ex (\ell,H)$, that is, the largest number of edges in an $\ell$-vertex graph that is $H$-free.  Observe first that an $(n,k)$-graph has at  least $\binom{n}{2}$ edges. If it is $H$-free, we must have that $\ex(nk, H)\geq \binom{n}{2}$.
We show that this inequality on $k$ given by a natural necessary condition can be almost attained by an $H$-free $(n,k)$-graph, provided good bounds for $\ex (\ell,H)$ are known.

\begin{thm}\label{main}
	Let $H$ be a bipartite graph that is not a forest and such that  for any sufficiently large $\ell$, $C\ell^a\leq \ex(\ell, H)\leq C'\ell^b$ for positive constants $C, C'$ and exponents $a, b$ with $b-a < \frac{(2-b)(b-1)}{5-b}$. 
	\begin{itemize}
		\item 	If $\ex (nk,H)  \geq  12n^2 \log n$, then $f(n,H) \leq k+ o(k)$. 
		\item 	If $\ex (nk, H) < \binom{n}{2}$, then $f(n, H) \geq k$.
	\end{itemize}
\end{thm} 
Thus, $f(n,H)$ is determined up to a logarithmic multiplicative factor. For graphs for which the Tur\'an exponents are known, there is a stronger result. \\

\begin{cor}\label{cor}
	If $H$ is a bipartite graph such that $\ex (\ell,H)=\Theta(\ell^r)$, then there are positive constants $C',C$ such that $C' n^{2/r-1} \leq  f(n, H) \leq Cn^{2/r-1} \log^{1/r} n$.
\end{cor}
In particular, Corollary \ref{cor}  applies to the following graphs $H$:     a complete bipartite graph  with parts of sizes $s$ and $t$: $H=K_{s,t},$    where $2=s\leq t$, or  $s=t=3$, or  $t>(s-1)!$, a cycle of length $4, 6,$ or  $ 10$. See a survey by F\"uredi and Simonovits~\cite{FS}. In addition, Corollary \ref{cor} is applicable  for  $K^k_{s,t}$, $s,k\geq 2$, $t$ sufficiently large, where  $F^k$ denotes a graph obtained from $F$ by replacing its edges with internally vertex-disjoint paths of length $k$, see Janzer \cite{J},  Conlon et al. \cite{C}, Jiang and Qiu \cite{JQ, JQ-2}; as well as several other graph classes, see Jiang et al. \cite{JMY}. 

In the case where $H=K_{2,t}$, we are able to eliminate the logarithmic term via an explicit construction using finite fields. 
\begin{thm}\label{K2t} 
	For each integer $t\geq 2$, there is a positive constant $C=C(t)$ such that for any sufficiently large $n$, $Cn^{1/3} \leq f(n, K_{2,t}) \leq (2+o(1))n^{1/3}$. In particular, $(1-o(1)) n^{1/3} \leq f(n, C_4)\leq (2+o(1)) n^{1/3}$.
\end{thm}

For trees, we have bounds that differ by a multiplicative constant of $4$.  
\begin{thm}\label{tree}
	Let $T$ be a tree on $t$ edges, $t\geq 2$. Then $\bigl\lceil\frac{n-1}{2(t-1)}\bigr\rceil \leq f(n,T) \leq \bigl\lceil\frac{2n}{t-1}\bigr\rceil$. Moreover, if $T$ is a star, then $\bigl\lceil\frac{n-1}{t-1}\bigr\rceil \leq f(n ,T) \leq \bigl\lceil\frac{n}{t-1}\bigr\rceil$. Hence if $T$ is a star, then $f(n,T)=\bigl\lceil\frac{n-1}{t-1}\bigr\rceil$ unless $t-1$ divides $n-1$. 
 \end{thm}

Section~\ref{sec:lemmas} contains general observations. Sections~\ref{sec:main},~\ref{sec:K2t}, and~\ref{sec:tree} contain the proofs of Theorem~\ref{main}, Theorem~\ref{K2t}, and Theorem~\ref{tree}, respectively. We discuss open problems in Section~\ref{sec:conc}. 

\section{General observations and auxiliary lemmas}
\label{sec:lemmas}

Let ${\rm r}(H; k)$ denote the $k$-color Ramsey number for $H$; that is, the smallest integer $n$ such that any coloring of the edges of $K_n$ in $k$ colors contains a monochromatic subgraph isomorphic to $H$.
For a graph $F$, we denote its vertex set and edge set by $V(F)$ and $E(F)$, respectively, and the maximum degree of $F$ by $\Delta(F)$. Let $H$ be a graph and $n \geq |V(H)|$ be an integer. If a graph $F$ contains no subgraph isomorphic to $H$, we say that $F$ is {\it $H$-free}. \\
	
\begin{lem}\label{LemGenN} 
	Let $H$ be a graph. If there exists an $(n,k)$-graph that is  $H$-free, then $\binom{n}{2} \leq  \ex (nk, H)$. Moreover there is an $H$-free $(n,k)$-graph  such that $n = {\rm r}(H; k)-1$.
\end{lem}

\begin{proof}
	Consider an  $(n,k)$-graph  $G$ that is $H$-free. Then  $|V(G)|= nk$,  $|E(G)| = \binom{n}{2}$,  and $|E(G)|\leq \ex (nk, H)$. So, the first claim of the lemma follows. 
			
	Let $n = {\rm r}(H;k)-1$ and fix a $k$-edge-coloring $c$ of $K_n$ on a vertex set $[n]$ with a color set $[k]$ that has no monochromatic copy of $H$. 
	We construct  an $(n,k)$-graph $G$ by  taking $k$ pairwise vertex disjoint isomorphic copies of the color classes of $c$.  Formally,  for each vertex $i$ of $K_n$ let the $i^{\rm th}$ part of $G$,  have vertices $v_i^1, v_i^2, \ldots, v_i^k$, $i\in [n]$.  For the $i^{\rm th}$ and $j^{\rm th}$ parts of $G$, we let the edge between them be $v_i^\ell v_j^\ell$, where   $\ell = c(ij)$,  $i,j\in [n]$, $i\neq j$. 
	Thus $G$ is a vertex-disjoint union of $k$ graphs, each of which is isomorphic to a different color class in $c$. Since each color class of $c$ is $H$-free,  $G$ is  also $H$-free.
\end{proof}
		
	We shall need the following lemma for our main result. It shows that for any bipartite graph $H$ and any sufficiently large $n$  there is an $n$-vertex graph that is $H$-free, has almost as many edges as in an extremal for $H$ graph, and has maximum degree not much higher as the average degree.

\begin{lem}\label{max-degree}
	Let $H$ be a bipartite graph and let $b$, $1<b<2$, be an upper Tur\'an exponent for $H$; that is, there is a constant $C'$ such that $\ex (\ell, H) \leq C'\ell^b$ for any $\ell$ sufficiently large.
	Then there exists an integer $\ell_0$ such that for any $\ell>\ell_0$ there is an $H$-free graph $G'$ on $\ell$ vertices with at least $\ex (\ell, H)/2$ edges and maximum degree at most $q\cdot \ex (\ell, H)/(2\ell)$, where $q = \max\left\{3,\left\lfloor C\left(\ell^b/\ex (\ell, H)\right)^{1/(b-1)}\right\rfloor\right\}$ and $C = (C')^{1/(b-1)} 2^{(b+1)/(b-1)+1}$. 
	
	In particular, when $\ex (\ell, H) = \Theta(\ell^b)$,  there is an $H$-free graph on $\ell$ vertices, $\ex (\ell,H)/2$ edges, and maximum degree at most a constant factor of the average degree.
\end{lem}

\begin{proof}
	We closely follow  the proof idea  of a similar result by Erd\H{o}s and Simonovits \cite[Theorem 1]{ES}. 
	Let $G$ be an extremal graph for $H$ on $\ell$ vertices, i.e., $G$ is an $H$-free graph with $\ex (\ell,H)$ edges. 
	Let $m=|E(G)|=\ex (\ell,H)$. 

	Split the  vertex set of $G$ into $q$, $q\geq 3$ parts $A_1, \ldots, A_q$ of nearly equal sizes such that each vertex in $A_1$ has degree at least as large as the degree of any vertex in  $V(G)-A_1$, i.e., $A_1$ is a set of $\lceil \ell/q\rceil$ vertices of the largest degrees. \\

	\noindent\textbf{Case 1.} 	$\sum_{x\in A_1}  d(x) \geq m/2$.\\
	The number of edges from $A_1$ to $V(G)-A_1$ is $\sum_{x\in A_1}  d(x)  -    2|E(G[A_1])|$.
	Then there is an index $j\in \{2, \ldots, q\}$ such that  the number of edges from $A_1$ to $A_j$  is at least 
	\begin{align*}
		\frac{1}{q-1}\left(\sum_{x\in A_1}  d(x) - 2 \left|E(G[A_1])\right|\right) .
	\end{align*}
	Thus
	\begin{eqnarray*}
		\left|E(G[A_1\cup A_j])\right|  &\geq &  \frac{\sum_{x\in A_1}  d(x)}{q-1}   - \frac{2|E(G[A_1])|}{q-1} + |E(G[A_1])| \\
		&\geq& \frac{m}{2(q-1)} + \frac{q-3}{q-1} |E(G[A_1])| \\
		&\geq& \frac{m}{2(q-1)} .
	\end{eqnarray*}
	On the other hand, if $\ell/q$ (hence, $\ell$) is sufficiently large, then 
 	\begin{align*}
		\left|E(G[A_1\cup A_j])\right| \leq \ex\left(2\lceil\ell/q\rceil,H\right) \leq C'\left(2\lceil\ell/q\rceil\right)^b < C' \frac{2^b \ell^b}{q^{b-1} (q-1)} .
	\end{align*}
 	Putting this all together gives
	\begin{align*}
		\frac{m}{2(q-1)} < C' \frac{2^b \ell^b}{q^{b-1} (q-1)} . 
	\end{align*}
	Solving for $q$ gives $q^{b-1} < C' 2^{b+1} (\ell^b/m)$, 
	a contradiction to our choice of $q$. \\
 
	\noindent\textbf{Case 2.} $\sum_{x\in A_1}  d(x) <m/2$. \\
	In particular there is a vertex in $A_1$ of degree at most $m/(2|A_1|) \leq qm/(2\ell)$.
	This implies that each vertex in $V(G)-A_1$ has degree at most $qm/(2\ell)$.
	In addition,  $E(G[V-A_1]) \geq  |E(G)| - \sum_{x\in A_1}  d(x) > m - m/2 = m/2$.
	The graph that is the union of $G[V-A_1]$ and $|A_1|$ isolated vertices satisfies the conditions of the lemma.
\end{proof}~\\

\section{Proof of Theorem \ref{main}}
\label{sec:main}

We shall obtain the desired $(n,k)$-graph by taking a nicely behaving $H$-free graph and randomly partitioning its vertex set.
For this, we need some definitions. 
	
We say that two random variables $X,Y$ are \textit{positively correlated} if $\E[XY]\geq\E[X]\E[Y]$ and a set of random variables are positively correlated if they are pairwise positively correlated. We shall apply a concentration inequality by Janson~\cite[Theorem 5]{Janson} which is one of a few variations on an inequality due originally to Suen~\cite{Suen} (see also~\cite[Theorem 8.7.1]{AlonSpencer}). We provide Janson's setup below:
	
\newtheorem{setup}[thm]{Setup}
\begin{setup}
	The probabilistic space for Theorem~\ref{thm:Janson} is as follows:
	\begin{itemize}
		\item 	$\{I_i\}_{i\in\mathcal{I}}$ is a finite family of indicator random variables defined on a common probability space.
		\item 	$\Gamma$ is a \textit{dependency graph} for $\{I_i\}_{i\in\mathcal{I}}$, i.e., a graph with vertex set $\mathcal{I}$ such that if $A$ and $B$ are two disjoint subsets of $\mathcal{I}$, and $\Gamma$ contains no edge between $A$ and $B$, then the families $\{I_i\}_{i\in A}$ and $\{I_i\}_{i\in B}$ are independent.
		\item 	$S = \sum_i I_i$. In particular $S=0$ if and only if all $I_i=0$. The expression $\Pr(S=0)$ in Theorem~\ref{thm:Janson} may thus be regarded as shorthand for $\Pr(\wedge_i \{I_i=0\})$. 
		\item 	$i\sim j$, where $i,j\in\mathcal{I}$, if there is an edge in $\Gamma$ between $i$ and $j$. (In particular, $i\not\sim i$.)
		\item 	$p_i = \Pr(I_i=1) = \E[I_i]$. Thus, $\Pr(I_i=0) = 1-p_i$.
		\item 	$\mu = \E[S] = \sum_i p_i$.
		\item 	$D = \sum_{\{i,j\} : i\sim j} \E[I_iI_j]$, summing over \textit{unordered pairs} $\{i,j\}$, i.e., over the edges in $\Gamma$. As a sum over ordered pairs, we have $D = \frac{1}{2} \sum_{i\in\mathcal{I}} \sum_{j\sim i} \E[I_iI_j]$. 
	\end{itemize}\label{setup:Janson}
\end{setup}
	
\begin{thm}[Janson~\cite{Janson}]\label{thm:Janson}
	Given the conditions in Setup~\ref{setup:Janson}, if the variables $\{I_i\}$ are positively correlated, then
	\begin{align*}
		\Pr(S=0) 	\leq 	\exp\left\{-\min\left\{\frac{\mu^2}{48D}, \frac{\mu}{4}\right\}\right\} .
	\end{align*}
\end{thm}

Given Theorem~\ref{thm:Janson} we are ready for the main result.
	
\begin{proof}[Proof of Theorem \ref{main}]
\vskip 0.3cm
\noindent
	Lemma~\ref{LemGenN} gives that $\ex (nk,H) < \binom{n}{2}$ implies $f(n,H) \geq k$. 
	
	For the other direction, choose $n$ sufficiently large and let $k$ be such that $12 n^2 \log n \leq \ex (nk, H)$.
	We shall show that there is an $(n, k)$-graph that is $H$-free. 
	
	Let $N=nk$. Recall that $CN^a\leq {\rm ex }(N, H) \leq C'N^b$, for positive constants $C, C'$ and for exponents $a, b$ such that ~$1<a<b<2$. Because $b-a<(2-b)(b-1)/(5-b)$, it is the case that
	\begin{align}\label{eq:abound}
		\frac{2(b+1)}{5-b}<a<b<2 .
	\end{align}
	Let $G$ be an $N$-vertex, $H$-free graph on $M=\ex (N,H)/2$ edges with maximum degree $\Delta= \Delta(G)\leq q M/N$ as guaranteed by Lemma~\ref{max-degree}, with $q = C\left(N^b/\ex (N, H)\right)^{1/(b-1)}$.  We shall show that the  vertex set of $G$ can be partitioned into $n$ parts of sizes at most $k+o(k)$  such that there is an edge between any two parts.  This would demonstrate that there is an $(n, k+o(k))$-graph that is $H$-free. 
	
	Let  $E(G)=\{e_1,\ldots,e_M\}$. We color each vertex of $G$ with a color from $\{c_1,\ldots,c_n\}$ independently and uniformly at random. For $i, j \in [n]$, let $S_{i,j}$ be the event that no edge gets colors $c_i$ and $c_j$ on its end vertices. 
	We shall show that  none of $S_{i,j}$ happen with positive probability, i.e., that with positive probability there is an edge between any two distinct color classes.
	
We start with $S_{1,2}$ first, i.e., we consider the probability that some edge gets colors $c_1$ and $c_2$ on its endpoints.	Let $\mathcal{I}=\{1, 2, \ldots, M\}$. For $i\in \mathcal{I}$ we define $I_i$ and $p_i$. We also define $\Gamma, \sim, \mu,$ and $D$ and give their basic properties:
		\begin{itemize}
			\item 	$I_i$ is the indicator random variable for the event that the set of colors on the endvertices of $e_i$ is  $\{c_1, c_2\}$.
			\item 	$\Gamma$ is a graph whose vertices are the edges of $G$. For distinct $i$ and $j$, we write $i\sim j$ if and only if edges $e_i$ and $e_j$ share an endvertex.
			\item 	$\Pr(S_{1,2}=0) = \Pr(\wedge_i \{I_i=0\})$ is the event that no edge gets colors $c_1$ and $c_2$. 
			\item 	$i\sim j$, where $i,j\in\mathcal{I}$, if there is an edge in $\Gamma$ between $i$ and $j$.
			\item 	$p_i = \Pr(I_i=1) = \E[I_i] = 2/n^2$. Thus, $\Pr(I_i=0) = 1-p_i = 1-2/n^2$.
			\item 	$\mu = \E[S_{1,2}] = \sum_i p_i = 2M/n^2$.
			\item 	$D = \sum_{\{i,j\} : i\sim j} \E[I_iI_j] = (2/n^3) \sum_{u\in V(G)} \binom{\deg(u)}{2} \leq N \Delta^2/n^3$. 
		\end{itemize}
	Moreover, the indicator variables $\{I_i\}$ are positively correlated. 
	That is, for all $i,j\in\mathcal{I}$,
	\begin{align*}
		\left(\frac{2}{n^2}\right)^2 
		= 	\E[I_i]\E[I_j] 
		\leq 	\E[I_iI_j]
		= 	\begin{cases}
				\frac{2}{n^3}, 			& \text{ if }e_i\cap e_j\neq\emptyset; \\
				\left(\frac{2}{n^2}\right)^2, 	& \text{ if }e_i\cap e_j=\emptyset.
			\end{cases}
	\end{align*}	
	
	Therefore, using Theorem \ref{thm:Janson}, we have
		\begin{align}\label{S12}
			\Pr(S_{1,2}=0) 	\leq 	\exp\left(-\min\left\{\frac{\mu^2}{48D}, \frac{\mu}{4}\right\}\right).
		\end{align}
	We will show that each of $\mu^2/(48D)$ and $\mu/4$ is at least $3\log n$.

	Recall that $\ex (N, H) \geq 12 n^2 \log n$ and ${\rm ex }(N, H) =\Omega(N^a)$, thus $N= O(\ex(N, H)^{1/a})$.
	In the following calculation, recall that $\Delta= \Delta(G)\leq q M/N$ with $q = C\left(N^b/\ex (N, H)\right)^{1/(b-1)}$:
	\begin{align*}
		\frac{\mu^2}{48D} 	&\geq  	\frac{(2 M / n^2)^2}{48 (N \Delta^2 / n^3)} \\
						&= 		\frac{M^2}{12 n N \Delta^2} \\	
											& \geq 	\frac{N}{12 n q^2} \\
						& \geq 	\frac {N}{12 n C^2 \left(N^b/\ex (N, H)\right)^{2/(b-1)}} \\
						&= 	 	\frac{1}{12 n C^2} \cdot \frac{(\ex (N, H))^{2/(b-1)}} {N^{(b+1)/(b-1)}} .
	\end{align*}
	
	Since $N= O(\ex(N, H)^{1/a})$, there is a constant $C''$ such that
	\begin{align*}
		\frac{\mu^2}{48D}	& \geq   	C'' \frac{1}{n} \ex (N, H)^{2/(b-1) - (b+1)/(a(b-1)) } \\
						& \geq  	C'' \frac{1}{n} n^{4/(b-1) - 2(b+1)/(a(b-1))} \log^{c}n \\
						& =  		C'' n^{[(5-b)- 2(b+1)/a]/(b-1)} \log^{c}n,
					\end{align*}
	for a constant $c \in (1/2,1/b)$.
	The exponent of $n$ is equal to 
	\begin{align*}
		\frac{(5-b)- 2(b+1)/a}{b-1} , 
	\end{align*}
	which is positive 
	as ensured by~\eqref{eq:abound}. In this case, for large $n$, this term dominates the logarithmic term and we have indeed that
	\begin{align*}
		\frac{\mu^2}{48D} \geq 3 \log n .
	\end{align*}
	
	Now, we consider the second term in (\ref{S12}),  $\mu/4$: 
	\begin{align*}
		\frac{\mu}{4} 	= 	\frac{2M/n^2}{4} 	\geq  \frac{\ex (N, H)}{4n^2} \geq \frac{12 n^2 \log n}{4 n^2} =  3 \log n.
	\end{align*}

	Using the union bound for probability,
	\begin{align*}
		\Pr\left(\exists \{i,j\}\in \textstyle{\binom{[n]}{2}} : S_{i,j}=0\right) 	&\leq 	n^2 \Pr(S_{1,2}=0) \leq \exp\left(2\log n - \min\left\{\frac{\mu^2}{48D}, \frac{\mu}{4}\right\}\right) \\
														&\leq 	\exp\left(2\log n - 3\log n\right) \leq n^{-1} .
	\end{align*}
	
	As a result, we have shown that with probability $1-n^{-1}$ there is a partition of vertices of an  $(nk)$-vertex graph $G$ into $n$ parts   with an edge between each pair of parts. It remains to show that no part is too big. 
	
	The expected number of vertices in each blob is $k$. Note that we can assume that $k>n^{\delta}$ for some positive $\delta$.
	Fix some $i\in[n]$ and let $X=X_i$ be the number of vertices of color $c_i$.
	We use a Chernoff-type concentration bound, \cite{MU}:
	\begin{align*}
		\Pr\left(\left|X-\E[X]\right| > \epsilon \E[X]\right) \leq \exp(-\epsilon^2 \E[X]/3) = \exp( -\epsilon^2 k/3) .
	\end{align*}
	
	Let $\epsilon = k^{-1/2} \log n$. Then taking the union bound over all $n$ color classes, we see that the probability that some color class deviates from the expected value of $k$ by more than $\sqrt{k}\log n = o(k)$ is at most $1/n$. Thus with probability at least $1 - 1/n -1/n$ each color class is of size at most $k+o(k)$ and there is an edge between any two distinct color classes.
\end{proof}~\\

\begin{rem}\label{rem:logterm}
	There was no attempt to optimize the constants in this proof because they are subsumed by the $\log^{1/a}n$ term and we conjecture that the logarithmic term can be removed. That is, we conjecture that there exists a constant $c$ such that if $\ex(nk,H) > cn^2$, then $f(n,H)\leq k$. 
\end{rem}

\begin{proof}[Proof of Corollary \ref{cor} ]
If $\ex(\ell, H) = \Theta(\ell^r)$, then $r$ plays a role of $a$ and $b$ in Theorem \ref{main}. In particular the inequality $a> 2(b+1)/(5-b)$ is satisfied.
Note also that since $H$ is not a forest and $H$ is bipartite $1<r<2$. 

Assume that $\ex (nk,H) \geq C (nk)^r$ for a positive constant $C$ and $nk$ sufficiently large.
Let $k= (12/C)^{1/r}n^{2/r-1} \log^{\frac{1}{r}} n = C' n^{2/r-1} \log^{1/r} n$.
Then $\ex (nk,H) \geq C (nk)^r   \geq  12n^2 \log n$, and  by Theorem \ref{main}  we have  $f(n,H) \leq  C'n^{2/r-1} \log^{1/r} n (1+o(1))\leq  2C' n^{2/r-1} \log ^{1/r} n $. 

On the other hand,  if $\ex(\ell, H) \leq C_1 \ell^r$  for $C_1>1$  and   $k= (C_1/2)^{-1} n^{2/r-1}$, we have that   $\ex (nk, H) \leq \frac{1}{2}C_1 C_1^{-r}  n^2 < \binom{n}{2}$, then 
by Theorem \ref{main} $f(n, H) \geq k = (C_1/2)^{-1} n^{2/r-1} = C''  n^{2/r-1}$.
\end{proof}

\section{$C_4$-free construction with $k=2n^{1/3}$}
\label{sec:K2t}

\begin{lem} \label{C4}
	Let $p$ be a prime power. There is a $(p^3,2p)$-graph that is $C_4$-free.  
\end{lem}
		
	\newcommand{\cP}{{\mathcal P}}
	\newcommand{\F}{\mathbb{F}}
	\newcommand{\Z}{\mathbb{Z}}
	
\begin{proof}
	We will construct the graph from what is known as the classical affine plane.  Let $q=p^2$.  In what follows,  the arithmetic is in $\mathbb{F}_q$. Let 		
	\begin{align*}
		\cP 	&= 	\left\{(x,y) : x \in \F_q; y \in \F_q\right\}  \text{ and }\\
		\cL 	&= 	\left\{\left\{(x, m x + b) : x \in \F_q\right\} : m \in \F_q; b \in \F_q\right\}.
	\end{align*}
	We refer to $\cP$ as the set of points and $\cL$ as the set of lines.
	For a line $\{\left\{(x, m x + b) : x \in \F_q\right\} $, we say that $m$ is its {\it slope} and $b$ is its {\it intercept}.
	Note that for any $m, b, b'\in \F_q$, $\{(x, m x + b) : x \in \F_q\} \cap \{(x, m x + b') : x \in \F_q \} = \emptyset$, thus lines with the same slope and distinct intercepts are disjoint. 		
	The set of all members of $\cL$ with the same slope is called a \textit{parallel class}. 
	For $m\in\F_q$, let the respective parallel class be denoted $\cL_m$, i.e., $\cL_m = \{\{m x + b : x\in\F_q\} : b\in\F_q\}$. 
	
	Let  $G=G(p)= (\cP,\cL;E)$  be a bipartite graph with parts $\cP$ and $\cL$, where $P\in\cP$ is adjacent to $L\in\cL$ if and only if $P\in L$. \\
		
\textbf{Claim 1.} {\it $G$ contains no copies of $C_4$.} \\
	\indent In order to prove this claim, it is sufficient to verify that  any distinct $L_1,L_2\in\cL$ have at most one common neighbor.
	Let $L_1,L_2$ have slopes $m_1,m_2$, respectively, and intercepts $b_1,b_2$, respectively. 
	Since the lines are distinct, $(m_1,b_1)\neq (m_2,b_2)$. 
	In the graph $G$, any common neighbor of $L_1$ and $L_2$ would be a pair $(x,y)$ such that $y=m_1 x + b_1=m_2 x+b_2$. Hence,
	\begin{align}
		(m_1 - m_2) x 	= 	b_2 - b_1. \label{eq:modulo}
	\end{align}
	If $m_1=m_2$ then $b_1=b_2$, a contradiction to the assumption that the lines are distinct. 
	If $m_1\neq m_2$, then $m_1-m_2$ has a multiplicative  inverse in  $\mathbb{F}_q$.
		Hence \eqref{eq:modulo} has a unique solution, namely $(m_1-m_2)^{-1}(b_1-b_2)$, so $L_1,L_2$ can have at most one common neighbor. This verifies Claim 1. \\
		
\textbf{Claim 2.}  {\it There is a partition of $V(G)$ into $p^3$ parts, each of size $2p$, such that there is an edge between any two distinct parts.} \\
	\indent For $x\in\F_q$, let $\cP_x = \{(x,y) : y\in\F_q\}$.  We see that  $\{\cP_x : x\in \F_q\}$ is a partition of $\cP$. 
	Furthermore, the set of parallel classes $\{\cL_m : m\in\F_q\}$ forms a partition of $\cL$ since each line in $\cL$ belongs to some parallel class and two distinct parallel classes are disjoint.  We further partition each $\cP_x$ and each $\cL_m$ as follows: 
	Let $H$ denote a subgroup of order $p$ of the group of $\F_q$ under addition. 
	Let $a_1,\ldots, a_p$ be members of distinct cosets of $H$ and let $A=\{a_1,\ldots,a_p\}$. 
	Note that $\{a+H : a\in A\}$ partitions $\F_p$.  For $x\in\mathbb{F}_q$, $h\in H$, $m\in\mathbb{F}_q$, and  $a\in A$, let
	\begin{align*}
			\cP_{x,h} 	= 	\left\{(x,y) : y\in A+h\right\} 	~~\mbox{ and } ~~
		\cL_{m,a} 	= 	\left\{\left\{(x, mx+b) : x\in\F_q\right\} : b\in a+H\right\} .
	\end{align*}
	
	Recall that since the $\cP_{x,h}$'s form a partition of $\cP$ and $\cL_{m, a}$'s form a partition of $\cL$,  their union is $\cP\cup \cL= V(G)$.
	Note that there are $p^3$ sets $\cP_{x,h}$ and $p^3$ sets $\cL_{m,a}$. 
 	Pair them up arbitrarily, take unions of the two sets in each pair, and call them $V_1, \ldots, V_{p^3}$.  Formally,  order all $\cP_{x,h}$'s and call them $\cP^*_1, \ldots, \cP^*_{p^3}$, 
	order all $\cL_{m,a}$'s  and call them $\cL^*_1, \ldots, \cL^*_{p^3}$, and let 
	$V_i = \cP_i^*\cup \cL_i^*$, $i=1, \ldots, p^3$.  Note that $|V_i|=2p$ for $i=1, \ldots, p^3$.
	We claim that $V_1, \ldots, V_{p^3}$ form a desired partition of the vertex set of $G$, i.e.,  there is an edge between any $V_i$ and any $V_j$, $1\leq i<j \leq p^3$. 
	In fact, in showing that there is an edge between any $V_i$ and $V_j$ for distinct $i$ and $j$, we show a stronger property that there is an edge between  $\cP_{x,h}$ and  $\cL_{m,a}$,  for any  $x,m\in\F_q$, $h\in H$, and $a\in A$. 
	
	In order for there to be an edge between $\cP_{x,h}$ and $\cL_{m,a}$, there must be $a_i\in A$ and $h_j\in H$ such that
	\begin{align*}
		a_i + h 	= 	mx + \left(a + h_j \right)  .
	\end{align*}
	Rearranging, we want to show there is an $a_i\in A$ and $h_j\in H$ such that
	\begin{align*}
		a_i - h_j 	= 	mx + a - h .
	\end{align*}
	Since sets $a_i + H = a_i - H = \left\{a_i - h_j : h_j\in H\right\}$ form distinct cosets for distinct $a_i$'s in $A$, the set $\{a_i + h_j : a_i\in A, h_j\in H\}$ is equal to $\F_q$ and so, no matter what $mx+a-h$ is, it must be equal to some $a_i+h_j$.  This verifies Claim 2.
	
Thus, $G$ is a $(p^3, 2p)$-graph that is $C_4$-free.
\end{proof}

\begin{proof}[Proof of Theorem \ref{K2t}]

	Let $n$ be a sufficiently large integer. 
	Let $N= n^{2/3}$. Note that there is an integer $k$, such that $N\leq k^2\leq N + 2\sqrt{N} $. 
	There is a prime $p$ such that $k\leq p\leq k + k^{0.6}$, \cite{BHP}. 
	Then $k^2\leq p^2 \leq (k+k^{0.6})^2$, i.e., $ N\leq p^2 \leq N + o(N)$.  
	Then $n= N^{3/2}$ satisfies $ n\leq p^3 \leq n +o(n)$. 
	
	Consider a $C_4$-free  graph $G$  that is a $(p^3, 2p)$-graph, as guaranteed by Lemma \ref{C4}.
	We see that $G$ contains an $(n, 2p)$-graph that is $C_4$-free. 
	Since $2p \leq (2 + o(1)) n^{1/3}$, it follows that $f(n, C_4)\leq (2+o(1)) n^{1/3}$. 

	To see the lower bound on $f(n, C_4)$, recall from Lemma~\ref{LemGenN} and the known upper bound on $\ex (n, C_4)$~\cite{FS}, that 
	\begin{align*}
		\binom{n}{2} \leq \ex (nk, C_4) \leq (1+o(1)) \frac{(nk)^{3/2}}{2} .
	\end{align*}
	This implies that $k \geq (1-o(1)) n^{1/3}$. 

	For the bounds on $f(n, K_{2,t})$ where $t\geq 3$, note that the upper bound argument gave a $C_4$-free, hence, $K_{2,t}$-free graph.
	The lower bound follows similarly using Lemma~\ref{LemGenN} and the known upper bound on $\ex (n, K_{2,t})$~\cite{FS},
$\ex (\ell, K_{2, t+1}) = (\frac{1}{2}+o(1)) \sqrt{t}\; \ell^{3/2}$. 
	Consequently, $k \geq (t^{-1/3}-o(1)) n^{1/3}$.
\end{proof}~\\

\section{Trees}	
\label{sec:tree}
%
%
%

\begin{proof}[Proof of Theorem \ref{tree}]
	Let $T$ be a tree on $t$ edges, $t\geq 2$.
	By Lemma~\ref{LemGenN}, we have that 
	\begin{equation}\label{eq}
	\max \biggl\{k\in \mathbb{N}:  \binom{n}{2} \geq  \ex (nk, T) +1 \biggr\}  \leq ~f(n,T)~ \leq  \min \biggl\{k \in \mathbb{N}: n\leq  {\rm r}(T;k)-1 \biggr\}. 
	\end{equation}
	Since $\ex (nk, T) \leq nk(t-1)$  and ${\rm r}(T; k)\geq (t-1)\lfloor (k+1)/2\rfloor +1$,   see \cite{ErdG} and  \cite{GRS}, we have 
	\begin{equation}\nonumber
	\max \biggl\{k\in \mathbb{N}:  \binom{n}{2} \geq  nk(t-1) +1 \biggr\}  \leq ~f(n,T)~ \leq  \min \biggl\{k \in \mathbb{N}: n\leq  (t-1)\biggl\lfloor \frac{k+1}{2}\biggr\rfloor \biggr\}
	\end{equation}
	and the claimed general  bounds on $f(n, T)$ hold.\\

	When $T$ is a star, we have $\ex (n, T)\leq n(t-1)/2$ and  ${\rm r}(T, k)=k(t-1) + \epsilon$, where $\epsilon=1$ if $k$ is even and $t$ is even, and $\epsilon=2$ otherwise, see \cite{BuRo1}. 
Using (\ref{eq}) and these more precise bounds on the extremal and Ramsey numbers,  we have that $\bigl\lceil\frac{n-1}{t-1}\bigr\rceil\leq f(n ,T) \leq \bigl\lceil\frac{n}{t-1}\bigr\rceil$.
\end{proof}~\\

\section{Conclusions}
\label{sec:conc}
We considered, for large $n$,  the smallest possible value of $k$ such that  there is an $H$-free graph that is a split of $K_n$, or in other words, an $(n,k)$-graph. 
We determined the order of such a graph in terms of the extremal number of $H$, for each graph $H$ that has a well-defined Tur\'an exponent to within a logarithmic term. 

Noting that $(n,k)$-graphs can be easily transformed into $K_n$-minors, we see that for some $H$, the obtained results give "small" $K_n$ minors that are $H$-free.
Indeed,  let  $H$ have minimum degree at least $3$. Consider a balanced  $(n,k)$-graph $G$  that is $H$-free,  with blobs $B_1, \ldots, B_n$. 
Consider pairwise disjoint sets $X_1, \ldots, X_n$ of new vertices and let $X_i\cup B_i$ induce a subdivision  $S_i$ of a star, i.e.   a spider  with legs of length at least $2$ and leaf sets $B_i$, $i=1, \ldots, n$.  Then the union of $G$ and all $S_i$'s  is still $H$-free and forms a $K_n$-minor. 

Finally, we would like to state the following conjecture, which would remove the logarithmic term from our bound: 
\begin{conj}
	There exists a positive constant $c$ such that if $\ex(nk,H) > cn^2$, then $f(n,H)\leq k$. 
\end{conj}~\\

\section*{Acknowledgements}

Martin's research was partially supported by a Simons Foundation Collaboration Grant (\#353292) and an award from the Deutscher Akademischer Austauschdienst (DAAD): Research Stays for University Academics and Scientists (Program 57442043) to visit the Karlsruhe Institute of Technology. Axenovich's research was partially supported by the Deutsche Forschungsgemeinschaft (DFG) grant (FKZ AX 93/2-1).

We thank Torsten Ueckerdt for introducing this problem to us, Tamar Mirbach for conversations on $f(n, C_4)$, and David Conlon for discussions on Tur\'an exponents. We also thank Craig Timmons for some informative discussions.

This manuscript has two corrections: Theorem \ref{tree}  has been missing floors and ceiling. It is corrected and the proof has been modified. We are indebted to Barnabas Janzer for pointing out this error. Additionally, in the third paragraph of the introduction, we state explicitly that $n<|V(H)|$ implies $f(n,H)=1$.

\newpage
\nocite{*}
\bibliography{references}{} 
\bibliographystyle{plain}
\newpage

\end{document}